\documentclass[a4paper,12pt]{article} 

\usepackage[utf8]{inputenc}
\usepackage[T1]{fontenc}

\usepackage{stmaryrd} 
\SetSymbolFont{stmry}{bold}{U}{stmry}{m}{n} 

\usepackage{amssymb,amsmath,amsthm,amsfonts}
\allowdisplaybreaks[2]

\usepackage[shortlabels]{enumitem}

\usepackage{array,multirow}
\usepackage{tabularray}

\usepackage{graphicx}

\usepackage{hyperref}
\hypersetup{
    colorlinks = true,
}
\usepackage{cleveref}

\theoremstyle{plain}
\newtheorem{thm}{Theorem}

\newtheorem{corol}{Corollary}
\newtheorem{lem}{Lemma}

\theoremstyle{definition}

\newtheorem{rem}{Remark}

\newtheorem*{rem*}{Remark}
\newtheorem*{rems*}{Remarks}

\theoremstyle{remark}

\newcommand\cD{{\mathcal D}}
\newcommand\cN{{\mathcal N}}
\newcommand\cM{{\mathcal M}}
\newcommand\N{{\mathbb N}}

\newcommand\R{{\mathbb R}}

\DeclareMathOperator{\E}{{\mathbb E}}

\DeclareMathOperator{\cov}{cov}

\newcommand{\CI}{\operatorname{CI}}

\DeclareMathOperator{\PP}{{\mathbb P}}

\renewcommand{\epsilon}{\varepsilon}

\DeclareMathOperator*{\weakcvge}{\Longrightarrow} %{\cvge^{(d)}}
\newcommand{\Lcvge}{\weakcvge_{n\to\infty}}

\newcommand{\as}{a.s.\@}
\newcommand{\iid}{i.i.d.\@}
\newcommand{\ie}{i.e.\@}
\newcommand{\wrt}{w.r.t.\@}

\newcommand{\posirpackage}{https://github.com/dbontemps/posir}

\title{Inference post region selection}
\author{
  Dominique Bontemps\thanks{Institut de Mathématiques de Toulouse, UMR5219; Université de Toulouse; CNRS; UPS, F-31062 Toulouse Cedex 9, France.}
  \and 
  François Bachoc\thanks{Institut de Mathématiques de Toulouse, UMR5219; Université de Toulouse; CNRS; UPS, F-31062 Toulouse Cedex 9, Institut universitaire de France (IUF), France.} % ANR ? FIXME
  \and 
  Pierre Neuvial\thanks{Institut de Mathématiques de Toulouse; UMR5219,
	Université de Toulouse; CNRS, 
    UPS, F-31062 Toulouse Cedex 9, France.}
}

\begin{document}

\maketitle

\begin{abstract}
Post-selection inference consists in providing statistical guarantees, based on a data set, that are robust to a prior model selection step on the same data set. In this paper, we address an instance of the post-selection-inference problem, where the model selection step consists in selecting a rectangular region in a spatial domain.
The inference step then consists in constructing confidence intervals on the average signal of this region.
This is motivated by applications such as genetics or brain imaging.
Our confidence intervals are constructed in dimension one, and then extended to higher dimension.
They are based on the process mapping all possible selected regions to their corresponding estimation errors on the average signal. We prove the functional convergence of this process to a limiting Gaussian process with explicit covariance. This enables us to provide confidence intervals with asymptotic guarantees. In numerical experiments with simulated data, we show that our coverage proportions are fairly close to the nominal level already for small to moderate data-set size. We also highlight the impact of various possible noise distributions and the robustness of our intervals. Finally, we illustrate the relevance of our method to a segmentation problem inspired by the analysis of DNA copy number data in cancerology.
\end{abstract}

MSC 62E20; 62F25, 62J15.

\section{Introduction}

A common practice in statistics is to select from data a model (defined for instance by a set of variables), which typically yields parameters to be estimated or hypotheses to be tested. 
However, most of the existing statistical methodology assumes that the model (and thus the parameters or hypotheses) are fixed. This gap between theory and practice induces a selection bias, which is addressed by the recent field of statistics called post-selection inference \cite{Berk13,taylor2015statistical}.

In this paper, we focus on a particular post-selection inference setting where a region is selected in a domain of interest, from data.
The goal is then to provide confidence intervals or tests for the average signal related to the choice of the selected regions. 
This setting arises in a number of applications: for instance, in genomic data analysis, for the detection of DNA copy number aberrations or differentially methylated regions; or, in neuroimaging data analysis, for the detection of brain regions associated to a specific cognitive task, see \cite{benjamini2018selection}.

Post-selection inference is an active field of statistics with many recent contributions.
A first strand of literature focuses on specific model selection procedures (for instance the Lasso \cite{tibshirani1996regression} for linear models) and provides confidence intervals or tests that are valid (in terms of coverage probability or type 1 error) conditionally to the selected model, see e.g. \cite{hyun2018post,lee15exact,tibshirani2018uniform}.
An approach specifically dedicated to the problem of conditional inference after region selection has been proposed by \cite{korthauer2019detection,benjamini2018selection}.

Another possible approach to address this problem consists in splitting the observations into two subsets, and use one for the selection of regions and the other one to perform inference on these regions. In our context, an easy way to implement this strategy in dimension one consists in taking the observations with even index for region selection, and those with odd index for inference. The downsides of this type of strategy are  a loss of statistical power and spatial resolution inherent to sample splitting, and the need to account for multiplicity at the inference stage. More recently, data thinning (or data fission) strategies consisting in splitting each individual observation have been explored in \cite{neufeld2024data-thinning,dharamshi2024generalized-data-thinning,leiner2025data-fission}. Such strategies generate two (or more) independent random data sets in such a way that, by applying a non-random operation (e.g the sum) to them, the initial dataset can be recovered. This idea can be applied to various distributions belonging to the exponential family. These strategies are intrinsically parametric, in the sense that the thinning/fission operation depends on the data distribution. As such, their robustness to deviations from the assumed distribution may be problematic. Moreover, they also depend on an hyperparameter which controls the proportion of information injected in each generated data set.

Another body of literature considers models that are selected arbitrarily \cite{Berk13,Kuchibhotla2020}. 
In a linear Gaussian regression setting, a model is defined as an arbitrary subset of covariates. For a family of admissible models, it is shown in \cite{Berk13} that a universal coverage property can be achieved by using a family of confidence intervals whose sizes are proportional to a so called Post-Selection Inference (PoSI) constant.
These results have been later extended to prediction problems \cite{bachoc2019valid} and to misspecified non-linear settings \cite{bachoc2020uniformly}.
The goal of the present paper is to obtain similar universal coverage properties for the problem of region selection, in the one-dimensional case (where regions of interest are typically intervals), with extensions to multi-dimensional cases (where regions of interest are hyper-rectangles). 
This is achieved by considering the process mapping all possible regions to the estimation error of their average signal. We prove that this process converges in distribution to an explicit Gaussian process, for which we can evaluate the quantiles of the supremum. This yields a practical method to compute confidence intervals, for which we can prove the asymptotic validity, in the post-region-selection context.

The article is organized as follows: \Cref{sec:1Dsettings} sets the one-dimensional framework in which we perform inference after the selection of regions, and \Cref{sec:cvrge_Wiener} presents our main results. \Cref{sec:1Dsimulations} shows some simulations in the 1D-setting, and \Cref{sec:illustration-segmentation} an illustration of the method for a segmentation problem. The proofs are given in \Cref{sec:proofs}. Then, \Cref{sec:higherdim} presents extensions to higher dimensions, including simulations in a two-dimensional setting in \Cref{sec:2Dsimulations}.

The proposed approach is implemented in the R package \texttt{posir}. This package is available at \url{https://github.com/dbontemps/posir}, together with scripts to reproduce the numerical experiments reported in the paper.

\section{Setting and notation} \label{sec:1Dsettings}

\subsection{Inference for one-dimensional regions}

Consider an observed random vector $Y$ of unknown mean $\mu\in\R^n$ :
\begin{equation} \label{eq:1Dmodel}
    Y = \mu + \epsilon,
\end{equation}
with a centered random error vector $\epsilon$, 
whose coordinates are supposed to be \iid{} with variance $\sigma^2$, but not necessarily Gaussian. 

We call region a connected subset $I$ of $\llbracket 1, n\rrbracket := \{1, 2, \ldots, n\}$ (in practice $I = \llbracket m_I+1, M_I\rrbracket$ is an interval since we work in dimension 1). We aim at estimating the regional mean
\[ \mu_I = \frac{1}{|I|} \sum_{i\in I} \mu_i, \]
where $|I| = M_I-m_I$ is the cardinal of $I$, 
using the corresponding empirical mean in the region $I$
\[ \widehat{\mu}_I = \frac{1}{|I|} \sum_{i\in I} Y_i. \]
Our goal is to suggest confidence intervals on $\mu_I$ based on $\widehat{\mu}_I$.
In practice, we will look for symetric intervals \wrt{} $\widehat{\mu}_I$ inspired from the Student $t$-statistics, \ie{}
\begin{equation} \label{eq:CI_shape}
  \CI_{I, \alpha} = \left[\widehat{\mu}_I \pm K_{1-\alpha}\, \frac{\widehat{\sigma}}{\sqrt{|I|}} \right],
\end{equation}
with $K_{1-\alpha}$ a $(1-\alpha)$-quantile to be defined later, and $\widehat{\sigma}^2$ a variance estimator.

We call model a collection $M=(I_k)_{1\leq k\leq K}$ of separate intervals. Let $\cM$ be a collection of such models, and suppose a selection procedure is used to select a model $\widehat{M}$ among $\cM$. We aim at obtaining simultaneous confidence intervals $\CI_{I, \alpha}$ for all $I\in \widehat{M}$ to estimate our targets $\mu_I$ in the selected model.

\subsection{PoSI approach}

We perform inference post region selection, so we aim for a guarantee that 
\begin{equation} \label{eq:aim_inf_post_sel}
  \Pr\big(\forall I\in \widehat{M}, \mu_I \in \CI_{I, \alpha}\big) \geq 1 - \alpha,
\end{equation}
for a given error level $\alpha>0$, without information on the selection procedure used to select $\widehat{M}$. 

A naive and commonly used way is to construct simultaneous confidence intervals $\CI_{M, j, \alpha}$ in all sub-models, so that
\begin{equation*} %\label{eq:rule_inf_post_sel}
  \sup_{M\in\mathcal{M}} \Pr\left(\bigcup_{I\in M} \big\{ \mu_I \notin \CI_{I, \alpha}\big\} \right) \leq \alpha.
\end{equation*}
However this approach is not sufficient to ensure that the simultaneous intervals have the desired error level in the selected model $\widehat{M}$. For instance a data-driven model selection procedure may have a tendency to actually select models in which one of the confidence intervals does not contain its target.

The PoSI approach promoted by \cite{Berk13} goes further, and (when adapted to our setting) demands that 
\begin{equation} \label{eq:PoSI_inf_post_sel}
  \Pr\left(\bigcup_{M\in\mathcal{M}} \bigcup_{I\in M} \big\{ \mu_I \notin \CI_{I, \alpha}\big\} \right) \leq \alpha.
\end{equation}

This relies on the upper bound 
\begin{equation} \label{eq:PoSI_core}
  \Pr\left(\bigcup_{I\in \widehat{M}} \big\{ \mu_I \notin \CI_{I, \alpha}\big\} \right) \leq \Pr\left(\bigcup_{M\in\mathcal{M}} \bigcup_{I\in M} \big\{ \mu_I \notin \CI_{I, \alpha}\big\} \right),
\end{equation}
which may seem overly conservative, but \cite{Berk13} proved that it is tight \emph{if no hypothesis is available on the selection procedure}. This was proved in the different setting of variable selection in linear regression, but the argument also applies here: one just has to actually choose an adversarial procedure that selects a model in which the simultaneous confidence intervals do not hold for relatively small values of $\alpha$.

Finally, a simplification is available with region selection: if a model $\widehat{M}=\big(\widehat{I}_k\big)_{1\leq k\leq \widehat{K}}$ has been selected among a collection $\cM$,
\begin{align}
  \Pr\left(\bigcup_{k=1}^{\widehat{K}} \big\{ \mu_{\widehat{I}_k} \notin \CI_{\widehat{I}_k, \alpha}\big\} \right)
  &\leq \Pr\left(\bigcup_{M\in\cM}\, \bigcup_{I\in M} \big\{ \mu_{I} \notin \CI_{I, \alpha}\big\} \right) \notag \\
  &= \Pr\left(\bigcup_{I : \exists M\in\cM, I\in M} \big\{ \mu_{I} \notin \CI_{I, \alpha}\big\} \right), \label{eq:PoSI_for_regions}
\end{align}
so it is enough to consider models composed of a unique interval $I$.

\subsection{Renormalizing the confidence intervals}

Going back to the form \eqref{eq:CI_shape} of our proposed confidence intervals, we can also write
\begin{align}
  \Pr\left(\bigcup_{M\in\cM}\, \bigcup_{I\in M} \big\{ \mu_{I} \notin \CI_{I, \alpha}\big\} \right)
  &= \Pr\left(\sup_{I : \exists M\in\cM, I\in M} \frac{\sqrt{|I|} \left|\widehat{\mu}_I-\mu_I\right|}{\sigma} > \frac{\widehat{\sigma}}{\sigma} K_{1-\alpha} \right) \notag \\
  &= \Pr\left(\sup_{(a, b)\in \cM} \left| D_{a, b}\right| > \frac{\widehat{\sigma}}{\sigma} K_{1-\alpha} \right), \label{eq:def_D_ab}
\end{align}
where 
\begin{equation*} 
  D_{a, b} = \frac{1}{\sigma \sqrt{b-a}} \sum_{j=a+1}^b \epsilon_j
\end{equation*}
and $(a, b)\in\cM$ stands for $(a, b)\in \llbracket 0, n\rrbracket^2 : \exists M\in\cM, I=\llbracket a+1, b\rrbracket\in M$.

This suggests that $K_{1-\alpha}$ should be a quantile associated to the distribution 
of $\displaystyle \sup_{(a, b)\in \cM} \left| D_{a, b}\right|$. This is done below in \Cref{sec:cvrge_Wiener} after we establish the asymptotics of this last distribution.

\subsection{Permitted regions and models}

Constraints will be imposed on the regions we consider. In particular, for some $\delta\in (0, 1]$, we consider all intervals $I$ whose cardinality $|I|$ is as least $\delta n$. Let us define
\[
  \cM_{n,\delta} = \left\{ (a,b) \in \llbracket 0, n \rrbracket^2 : b-a \geq \delta n \right\},
\]
so that the condition on $I = \llbracket m_I+1, M_I\rrbracket$ becomes $(m_I, M_I) \in \cM_{n,\delta}$.

Let us also define
\begin{equation}\label{eq:def_E_delta}
  E_\delta = \{ (s, t)\in [0, 1]^2 : t-s \geq \delta \}.
\end{equation}

\section{Asymptotics} \label{sec:cvrge_Wiener}

Following \eqref{eq:def_D_ab}, we need to compute the quantiles of 
$\sup_{(a,b) \in \cM_{n,\delta}} \left| D_{a,b} \right|$. Our strategy is to view this quantity as a function of a well-known random process. 

Let us define, for $t\in [0, 1]$, 
\begin{equation} \label{eq:def_B_tilde}
  \widetilde{B}_t^n = \frac{1}{\sigma\sqrt{n}} \sum_{i=1}^{\lfloor t n \rfloor} (Y_i-\mu_i),
\end{equation}
and, for any $\delta\in (0, 1]$, let $G_\delta$ be defined by
\begin{equation} \label{eq:def_G_delta}
  G_\delta : \cD([0, 1]) \to \R^+,\, w \mapsto \sup_{(s,t)\in E_{\delta}} \left|\frac{w_t - w_s}{\sqrt{t-s}} \right|,
\end{equation}
where $\cD([0, 1])$ is the space of cadlag real functions on $[0, 1]$, equipped with the supremum norm, as defined for instance in \cite{Billingsley68}.

\begin{lem} \label{lem:D_ab_en_fct_de_B_tilde}
  Let $\delta\in (0, 1)$ be fixed. We have 
  \[ \sup_{(a,b) \in \cM_{n,\delta}} \left| D_{a,b} \right| = G_{\widetilde{\delta}_n}(\widetilde{B}^n), \]
  where $\delta\leq \widetilde{\delta}_n :=  \dfrac{\lceil \delta n \rceil}{n}<\delta+\dfrac 1n$.
\end{lem}

\begin{thm}[Asymptotic distribution] \label{thm:D_ab_asymptotics}
  Suppose that the errors $(\epsilon_j)_{1\leq j\leq n}$ are centered, \iid{} with variance $\sigma^2$. Then
  \[ \sup_{(a,b) \in \cM_{n,\delta}} \left| D_{a,b} \right| %= G_{\widetilde{\delta}_n}(\widetilde{B}^n) 
  \Lcvge G_\delta(B), \]
  where $B$ is a Brownian motion on $[0, 1]$.
\end{thm}
Here and in the following, $\displaystyle\Lcvge$ denotes the weak convergence.

\Cref{thm:D_ab_asymptotics} relies on Donsker's theorem, recalled below in \Cref{sec:proofs}. 
\emph{Note that we do not need the normality of $\epsilon$.} 
There exist also variants of Donsker's theorem \emph{without the independence} of the random variables $\epsilon_i$, $i\geq 1$: see for instance \cite[theorem 20.1]{Billingsley68}, where $\epsilon$ is a stationary stochastic process with strong mixing conditions.
Relaxing homoscedasticity should also be possible, for instance by adapting the time grid. 

From now on, the process 
\[ (0, 1] \to [0, +\infty),\, \delta \mapsto \sup_{(a,b) \in \cM_{n,\delta}} \left| D_{a,b} \right| = G_{\widetilde{\delta}_n}(\widetilde{B}^n) \]
is referred to as a discrete (1D) POSIR process, which stands for PoSI-Regions (as a tribute to \cite{Berk13}). The limit process appearing in \Cref{thm:D_ab_asymptotics},
\[ (0, 1] \to [0, +\infty),\, \delta \mapsto G_\delta(B) = \sup_{(s,t)\in E_{\delta}} \left|\frac{B_t - B_s}{\sqrt{t-s}} \right| \]
is called continuous, or asymptotic, (1D) POSIR process. By extension, the POSIR method denotes the following construction of simultaneous confidence intervals based on the quantiles of the continuous POSIR process.

\begin{thm}[Simultaneous asymptotic confidence intervals]  \label{thm:simultaneous_asymptotic_ICs}
  Let $B$ be a Brownian motion on $[0, 1]$. For any $\alpha\in (0, 1)$ and $\delta\in (0, 1]$, let $K_{1-\alpha,\delta}$ be the quantile of order $1-\alpha$ of the distribution of $G_\delta(B)$, so that
  \[ \Pr\left( G_\delta(B) > K_{1-\alpha,\delta} \right) = \alpha. \]
  For any interval $I=I_{a,b}=\llbracket a+1, b\rrbracket$ with $a<b\in\llbracket 0, n\rrbracket$, define the following confidence interval for $\mu_{I} = \frac{1}{b-a} \sum_{j=a+1}^b \mu_j$
  \begin{equation} \label{eq:simultaneous_asymptotic_ICs}
    \CI_{I,\alpha,\delta} = \left[ \frac{1}{b-a} \sum_{j=a+1}^b Y_j \pm K_{1-\alpha,\delta}\, \frac{\widehat{\sigma}}{\sqrt{b-a}} \right],
  \end{equation}
  where $\widehat{\sigma}$ is an estimator of $\sigma$. 
  Further assume that either $\sigma$ is known and $\widehat{\sigma}=\sigma$, or that $\displaystyle \widehat{\sigma} \mathop{\longrightarrow}\limits_{n\to\infty}^{\PP} \sigma$. Then 
  \begin{equation} \label{eq:asympt:valid:sigma:connu}
  \lim_{n \to \infty} \Pr \left(\exists (a,b) \in \cM_{n,\delta} : \mu_{I} \not \in \CI_{I,\alpha,\delta} \right) = \alpha. 
  \end{equation}
  Alternatively, if we have instead $\forall c>0,\, \Pr\big(\widehat{\sigma}<\sigma-c\big) \mathop{\longrightarrow}\limits_{n\to\infty} 0$, then 
 \begin{equation} \label{eq:asympt:valid:sigma:inconnu}
 \limsup_{n \to \infty} \Pr \left(\exists (a,b) \in \cM_{n,\delta} : \mu_{I} \not \in \CI_{I,\alpha,\delta} \right) \leq \alpha. 
 \end{equation}
\end{thm}

\begin{corol}[Simultaneous asymptotic confidence intervals for the selected regions] \label{corol:simultaneous_selected_asymptotic_ICs}
  Under the same hypotheses as \Cref{thm:simultaneous_asymptotic_ICs}, consider the collection of all models $M=(I_k)_{1\leq k\leq K}$ whose intervals $I_k = \llbracket a_k+1, b_k\rrbracket$ are separate and satisfy $(a_k, b_k)\in\cM_{n,\delta}$ for all $k$. Then, for any selection procedure of $\widehat{M}=\big( \widehat{I_k} \big)_{1\leq k\leq \widehat{K}}$ among such models,
  \[ \limsup_{n \to \infty} \Pr \left(\exists k\in \llbracket 1, \widehat{K}\rrbracket : \mu_{\widehat{I_k}} \not \in \CI_{\widehat{I_k},\alpha,\delta} \right) \leq \alpha. \]
\end{corol}

\begin{rem} \label{rem:overestimationofsigma}
In \Cref{thm:simultaneous_asymptotic_ICs}, the first conclusion \eqref{eq:asympt:valid:sigma:connu} holds under a known $\sigma$ or a consistent estimator $\widehat{\sigma}$. In this case the asymptotic probability of uniform coverage is exactly $1 - \alpha$, which is ideal. 
This favorable setting may occur under structural assumptions on the mean vector $\mu$ that would be exploited in the definition of  $\widehat{\sigma}$. 

Nevertheless, if no assumption on $\mu$ is made at all, then a consistent estimator $\widehat{\sigma}$ is impossible to construct. This can be shown similarly as in Proposition A.3 in \cite{bachoc2020uniformly}. 
In contrast, an asymptotically conservative estimator $\widehat{\sigma}$ is always possible to construct, for which the second conclusion \eqref{eq:asympt:valid:sigma:inconnu} of  \Cref{thm:simultaneous_asymptotic_ICs} applies.
In this case the asymptotic probability of uniform coverage is at least $1 - \alpha$, which becomes conservative. 
An idea to construct an asymptotically conservative estimator $\widehat{\sigma}$  is simply to take the empirical standard deviation of the observation vector $Y$: heuristically, we note that the presence of the deterministic mean $\mu$ causes an overestimation of $\sigma$. 
In addition, the weaker the signal in $\mu$ (the variability of the components of $\mu$), the closer $\widehat{\sigma}$  would be to $\sigma$. 
We refer to Proposition 2.6 and Lemma C.4 in \cite{bachoc2020uniformly} for more details, and also to their Section A.2.1. 
\end{rem}

\section{Simulations} 
\label{sec:1Dsimulations}

\subsection{Quantiles} \label{sec:quantiles}

We present here estimations of the quantiles of the 1D asymptotic POSIR process $G_\delta(B)$ (with $B$ a standard Brownian motion), as defined in \eqref{eq:def_G_delta} and \Cref{thm:simultaneous_asymptotic_ICs}. 
The estimations of these quantiles rely on \Cref{thm:D_ab_asymptotics}. 

We simulated $10^6$ trajectories of a Gaussian white noise with a discretization step size of $1/n$ with $n=50000$. For each trajectory we computed $\sup_{(a,b) \in \cM_{n,\delta}} \left| D_{a,b} \right|$ for a grid of values of $\delta\in [0.005, 1]$. Then, for each value of $\delta$, we computed the empirical quantiles of order $1-\alpha$ for a grid of values of $\alpha\in [0.001, 0.5]$. 

\begin{table}[hptb]
  \begin{center}
    \small
    \begin{tabular}{|r|cccccccccc|}
      \hline
      $\alpha\backslash\delta$ & 1 & 0.9 & 0.8 & 0.7 & 0.6 & 0.5 & 0.4 & 0.3 & 0.2 & 0.1 \\ 
      \hline
      .5 & 0.675 & 1.128 & 1.338 & 1.517 & 1.687 & 1.855 & 2.029 & 2.231 & 2.485 & 2.853 \\ 
      .2 & 1.282 & 1.726 & 1.927 & 2.097 & 2.256 & 2.415 & 2.582 & 2.767 & 2.992 & 3.318 \\ 
      .1 & 1.645 & 2.086 & 2.281 & 2.444 & 2.595 & 2.746 & 2.903 & 3.077 & 3.287 & 3.588 \\ 
      .05 & 1.959 & 2.396 & 2.585 & 2.743 & 2.890 & 3.035 & 3.184 & 3.349 & 3.547 & 3.824 \\ 
      .01 & 2.577 & 3.000 & 3.189 & 3.335 & 3.469 & 3.601 & 3.736 & 3.883 & 4.059 & 4.303 \\ 
      .005 & 2.806 & 3.232 & 3.411 & 3.553 & 3.690 & 3.816 & 3.946 & 4.085 & 4.254 & 4.487 \\ 
      .001 & 3.293 & 3.715 & 3.885 & 4.009 & 4.140 & 4.251 & 4.372 & 4.502 & 4.659 & 4.879 \\
      \hline
    \end{tabular}
  \end{center}
  \caption{Some empirical quantiles for the (1D) POSIR process $G_\delta(B)$.}
  \label{table:quantiles_1D}
\end{table}

\Cref{table:quantiles_1D} shows a subset of the estimated quantiles. As expected, the quantiles are decreasing both in $\alpha$ and in $\delta$. Note that for $\delta=1$, we retrieve reasonable Monte-Carlo estimates of the quantiles of $|U|$ when $U\sim\cN(0, 1)$.

We also observe here a phenomenon already noted in \cite{Berk13} in a different context. For error levels $\alpha$ reasonably small ($\alpha\leq 5\%$), the quantiles of $G_\delta(B)$ differ from the standard Gaussian quantiles (case $\delta=1$) only by a moderate multiplicative factor (smaller than $2$ even for $\delta$ as small as $0.1$).
This suggests that the price to pay to obtain uniform (PoSI-type) guarantees is relatively small.

\subsection{Effective confidence levels} \label{subsection:effective:confidence}

We now present Monte-Carlo estimations of the effective simultaneous error levels of the confidence intervals \eqref{eq:simultaneous_asymptotic_ICs} proposed in \Cref{thm:simultaneous_asymptotic_ICs}. 
These estimations have been obtained for several values of $n$ and for various centered non-Gaussian distributions for the \iid{} random variables $\epsilon_1, \ldots, \epsilon_n$. We present below different series of such simulations. 
In each case, even if the standard deviation $\sigma$ is known since we generate the data, to compute the confidence intervals we used the usual estimator $\widehat{\sigma}$, with
\[
\widehat{\sigma}^2 = \frac{1}{n-1} \sum_{k=1}^n \left(\epsilon_k-\frac 1n \sum_{j=1}^n \epsilon_j\right)^2,
\]
since estimating $\sigma$ is needed in practice.
Our simulations tackle the case where there is no signal: here $\mu=0$.
While this choice may be surprising, its justification is that, with the estimator $\widehat{\sigma}$, any non-constant signal will introduce a positive bias on $\widehat{\sigma}$ (see \Cref{rem:overestimationofsigma}), which will increase the width of the confidence intervals and decrease their error level. Therefore, a process $Y$ with no signal is the least favorable case for evaluating the effective error levels of our confidence intervals.

On the basis of \eqref{eq:def_D_ab}, we computed 
\begin{equation} \label{eq:computed_levels}
  \frac{\sigma}{\widehat{\sigma}} \sup_{(a, b)\in \cM_{n,\delta}} \left| D_{a, b}\right| = \sup_{(a, b)\in \cM_{n,\delta}} \left| \frac{1}{\widehat{\sigma} \sqrt{b-a}} \sum_{j=a+1}^b \epsilon_j \right|
\end{equation}
for a large number of trajectories and a grid of values of $\delta$, and checked whether it was smaller than the quantiles estimated in \Cref{sec:quantiles}. A simple empirical proportion of negative answers gives a Monte-Carlo estimate of the effective simultaneous error levels of the confidence intervals. In all simulations below, we each time used $10^6$ independent trajectories of \eqref{eq:computed_levels} to compute these estimates.

The first series of simulations involve a Laplace distribution with parameter $\lambda=1$ (for the \iid{} random variables $\epsilon_k$). This provides an illustration of a case with moderately heavy tails. The effective confidence levels are presented in \Cref{fig:eff-conf-level_laplace} for various discretization levels $n$.

\begin{figure}[hptb]
  \includegraphics[width=\textwidth]{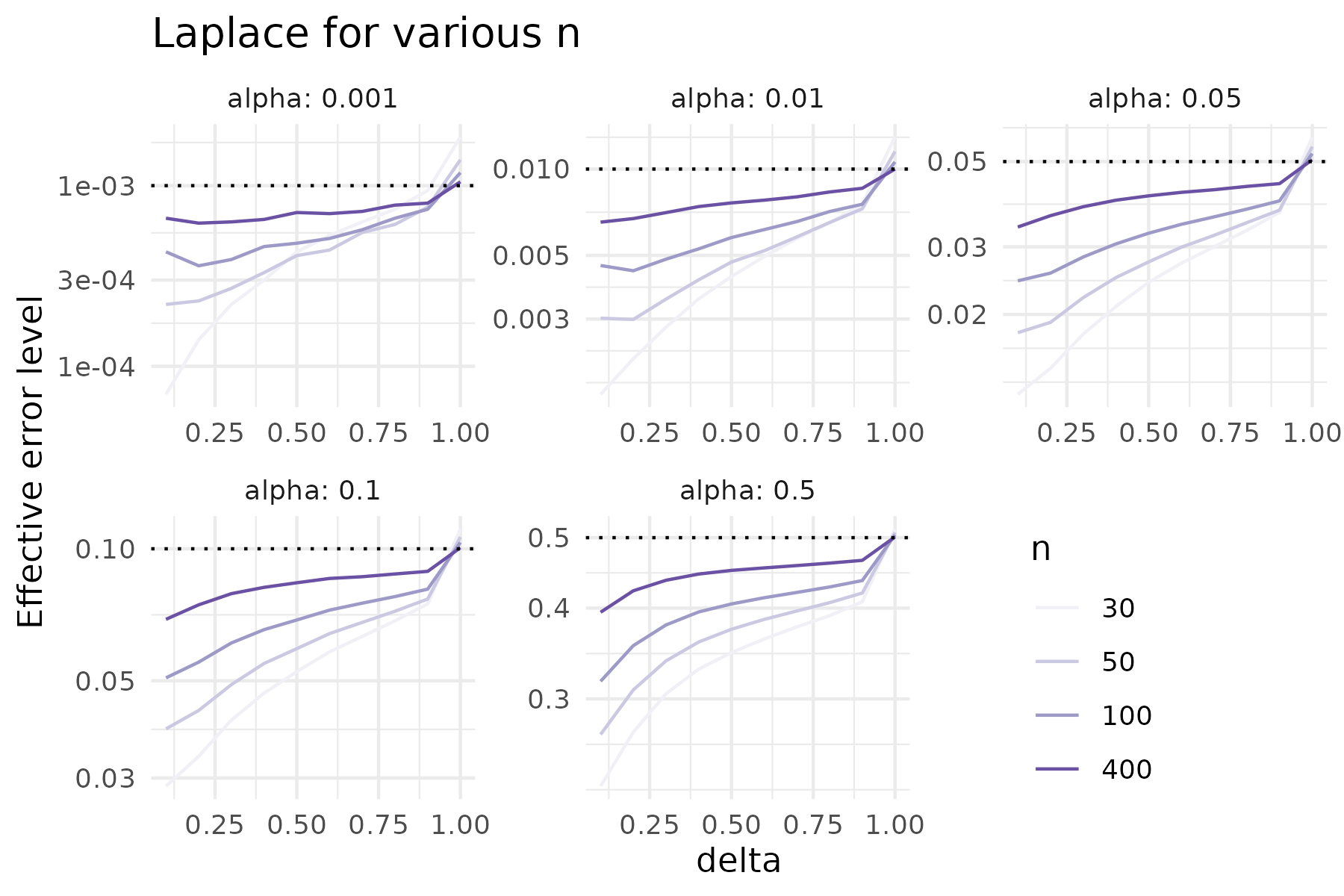}
  \vspace*{-1cm}
  \caption{Effective simultaneous error levels for the Laplace distribution with parameter $\lambda=1$ and various $n$.}
  \label{fig:eff-conf-level_laplace}
\end{figure}

First, the effective error levels are generally below the nominal error levels $\alpha$, indicating valid (but possibly conservative) confidence intervals. This can be explained by \Cref{lem:D_ab_en_fct_de_B_tilde}: the supremum in \eqref{eq:computed_levels} is a maximum over a finite number of possibilities, and it is naturally smaller than %$G_\delta(B)$, where 
when the supremum is taken over the limit continuous process.

Moreover, the gap between effective and nominal error levels $\alpha$ is larger when $\delta$ is very small, especially if $n$ is also small. This is not surprising: with $n=30$ and $\delta=0.1$ for instance, we compute the supremum over a moving window in which we may add only up to $3$ random variables $\epsilon_k$, each with a heavier tail that a normal variable, so we are very far from the Gaussian asymptotics.

Last, note that as $n$ goes larger, the gap between effective and nominal error levels reduces, which is consistent with Theorems~\ref{thm:D_ab_asymptotics} and \ref{thm:simultaneous_asymptotic_ICs}.

The second series of results involve a centered (but not symmetrized) Pareto distribution with shape parameter $2.1$. We emphasize that this corresponds to extremely heavy tails: below a shape parameter of $2$, the Pareto distribution has no quadratic moments. The results are summarized in \Cref{fig:eff-conf-level_centered-pareto_2.1}.

\begin{figure}[hptb]
  \includegraphics[width=\textwidth]{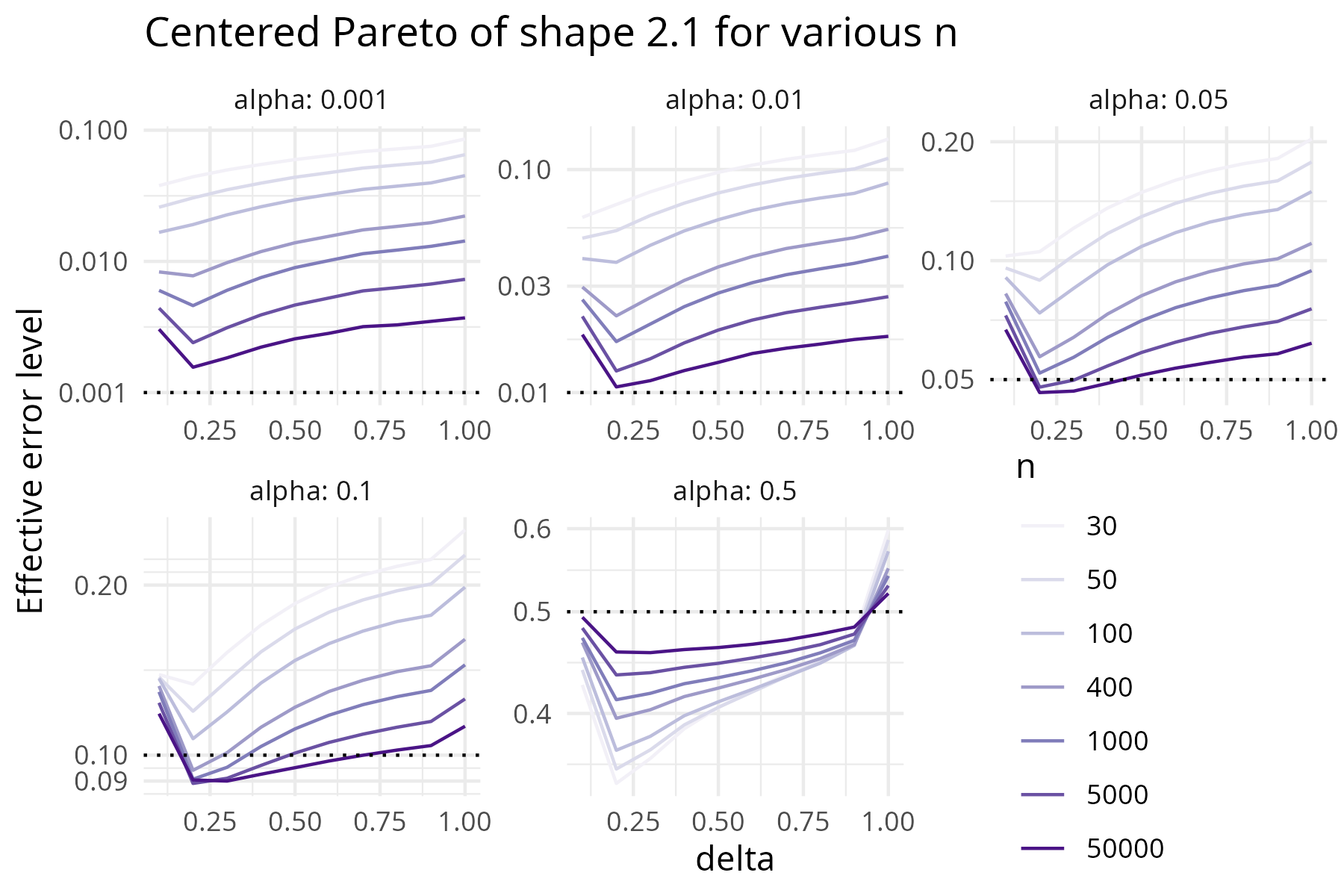}
  \vspace*{-1cm}
  \caption{Effective simultaneous error levels for the Pareto distribution with shape parameter $2.1$ and various $n$.}
  \label{fig:eff-conf-level_centered-pareto_2.1}
\end{figure}

The heavy tail of the Pareto (2.1) distribution implies that $\sup_{(a,b) \in \cM_{n,\delta}} \left| D_{a,b} \right|$ gets an heavier tail than $G_\delta(B)$. 
The Gaussian asymptotics, on which the confidence intervals of \Cref{thm:simultaneous_asymptotic_ICs} rely, are slow to appear.
Further, the \emph{relative} error between the effective error levels and the nominal level $\alpha$ is larger when $\alpha$ is very small. This is somewhat anticipated with heavy tails. 

\bigskip
We have also investigated the case of heavy-tailed but less pathologic distributions. These simulations are summarized by \Cref{fig:eff-conf-level_centered-paretos_n100}, in which we fix the discretization $n$ at $100$, and study various parameter shapes for the Pareto distribution. The results, presented in \Cref{fig:eff-conf-level_centered-paretos_n100}, show that the effective error levels get much closer to the nominal level as the tails become lighter. For completeness, we also present in Appendix (\Cref{fig:eff-conf-level_centered-pareto_3}) the result of a simulation with a a Pareto distribution with a shape parameter equal to $3$.

\begin{figure}[hptb]
  \includegraphics[width=\textwidth]{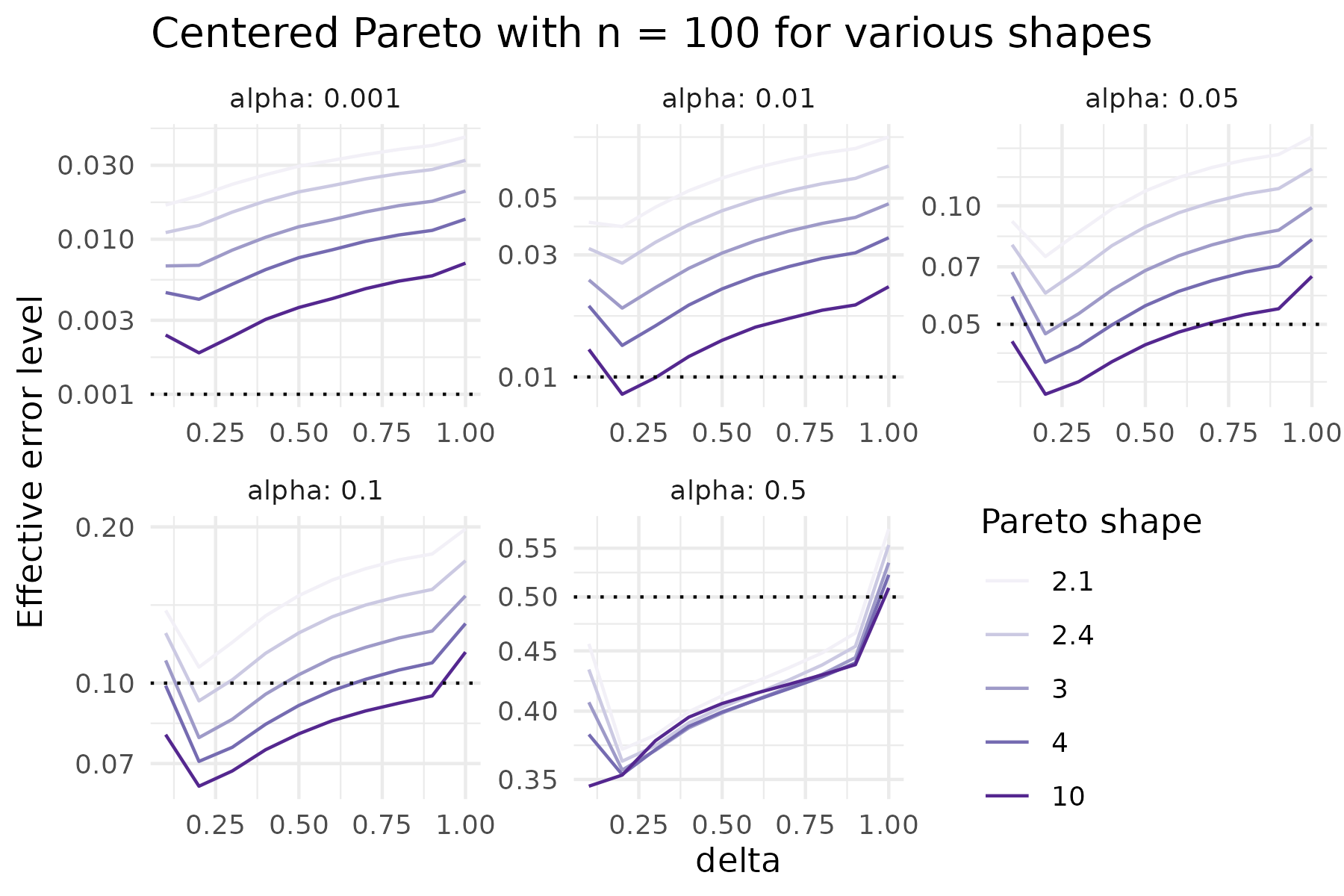}
  \vspace*{-1cm}
  \caption{Effective simultaneous error levels at $n=100$ for the Pareto distribution with various shape parameters.}
  \label{fig:eff-conf-level_centered-paretos_n100}
\end{figure}

%\bigskip

\section{Application to a segmentation problem}
\label{sec:illustration-segmentation}

We consider data that follows the model \eqref{eq:1Dmodel}, with a piecewise constant signal $\mu$ according to an unknown segmentation. We aim at providing simultaneous confidence intervals for the bin means of the signal in any given estimated segmentation, obtained by your preferred method of break detection.

Each bin or region is actually an interval. We limit ourselves to regions whose relative lengths are at least $\delta$, for a predetermined parameter $\delta\in (0, 1)$. Note that it may occur that a given breakpoint detection method returns a segmentation where some intervals are smaller than $\delta$. For such regions which are too small, we do not compute a confidence interval. 

In this section we show a practical illustration of the use of POSIR, compare it to a data splitting approach, and suggest applications that can be further built on our work.

\subsection{Simultaneous confidence intervals on segment means}
\label{sec:sim-ci-segment-means}

We use data generated by a resampling-based method introduced in \cite{PierreJean2015} in the context of DNA copy number studies in cancer research. Due to the tumoral process, cancer cells often show variation in the number of DNA copies along their genomes. This results in a piecewise-constant signal along chromosomes, each piece corresponding to a specific number of DNA copies. The method of \cite{PierreJean2015} aims to generate realistic DNA copy number profiles of cancer samples with known truth: given a simulated segmentation, the data in each segment is generated by re-sampling real copy number data in a known, segment-specific, cancer state. 

In \Cref{fig:posir-CI_true-breakpoints} we generated data with 10000 observations and 10 irregularly distributed real breakpoints, which are highlighted by red vertical lines. The location of these breakpoints is a parameter of the data generation process, which remains unknown for the subsequent data analysis. These data follow the model \eqref{eq:1Dmodel}, with a piecewise constant signal $\mu$ according to an unknown segmentation. The signal corresponds to a succession of three types of biological regions: normal regions with two DNA copies, and abnormal regions with either one or three DNA copies. For biological and technical reasons (in particular due to the presence of a mix of tumoral and normal cells), the observed signals are not calibrated around the true biological copy numbers but shifted toward lower values.

\begin{figure}[htb]
  \begin{center}
    \includegraphics[width=\textwidth]{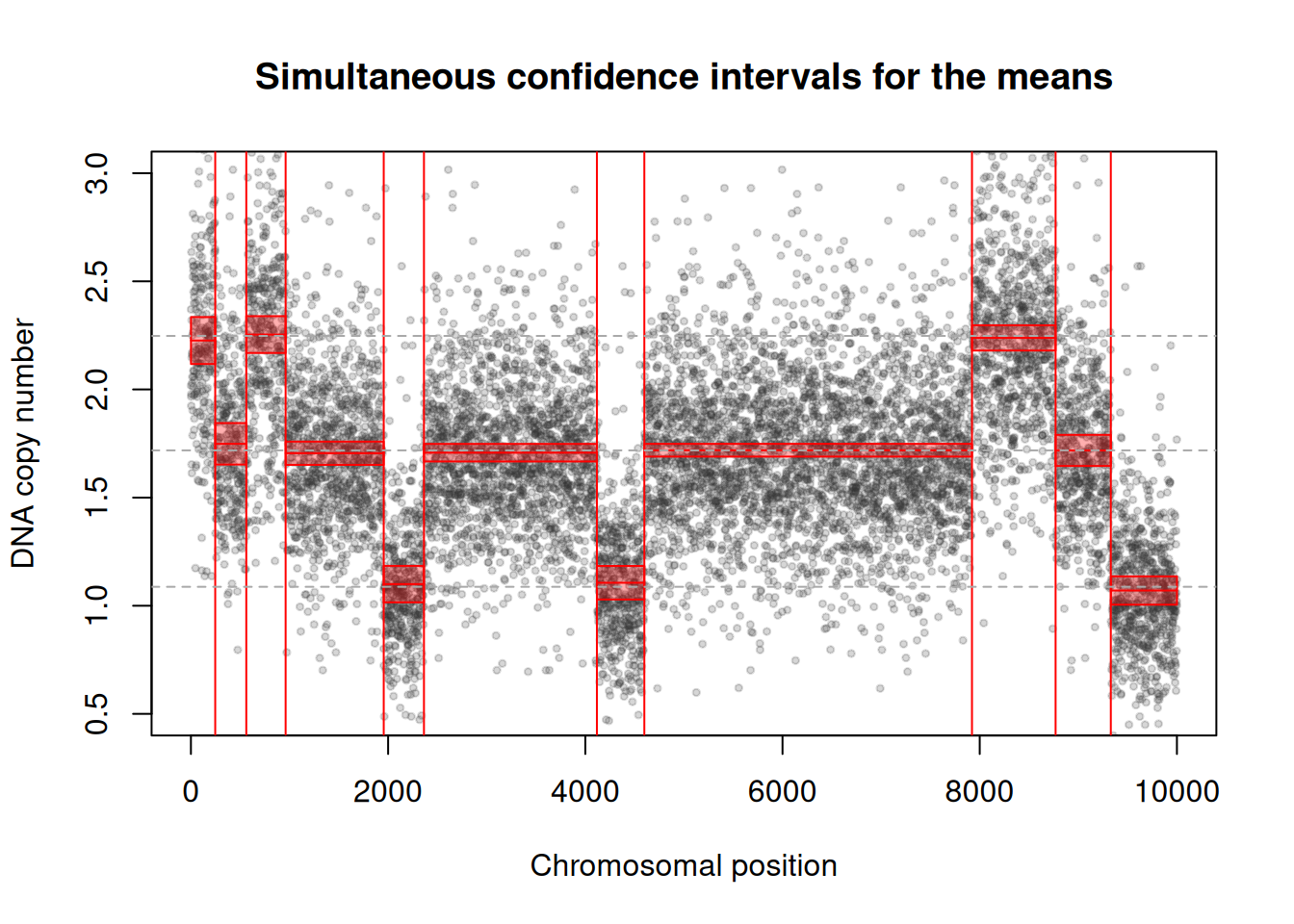}
  \end{center}
  \caption{Illustration of POSIR on DNA copy number data. Each dot corresponds to a genomic position on a chromosome. The 10 true breakpoints are displayed as vertical red lines, and the empirical means within each segment are also shown in red. All POSIR confidence intervals (displayed as red boxes) cover the true population means (horizontal dashed lines). 
} 
  \label{fig:posir-CI_true-breakpoints}
\end{figure}

The POSIR method provides simultaneous confidence intervals for the bin means of the signal for any genomic interval. The only restriction of our method is in the specification of the parameter $\delta \in (0,1)$: only intervals whose relative length is at least $\delta$ are considered. Here, we set $\delta=0.005$. In \Cref{fig:posir-CI_true-breakpoints}, these intervals are defined by the true breakpoints. As expected, all intervals cover the true population mean in this example. 

In practical applications, the true breakpoint locations are of course unknown, and so is the true number of breakpoints. For illustration, we show the results of a segmentation in 20 breakpoints, obtained by dynamic programming \cite{Rigaill2015} in \Cref{fig:posir-CI_segmentation}. Some breakpoints are false positives (e.g. around position 3200, and between positions 4500 and 8000). 
\begin{figure}[htb]
  \begin{center}
    \includegraphics[width=\textwidth]{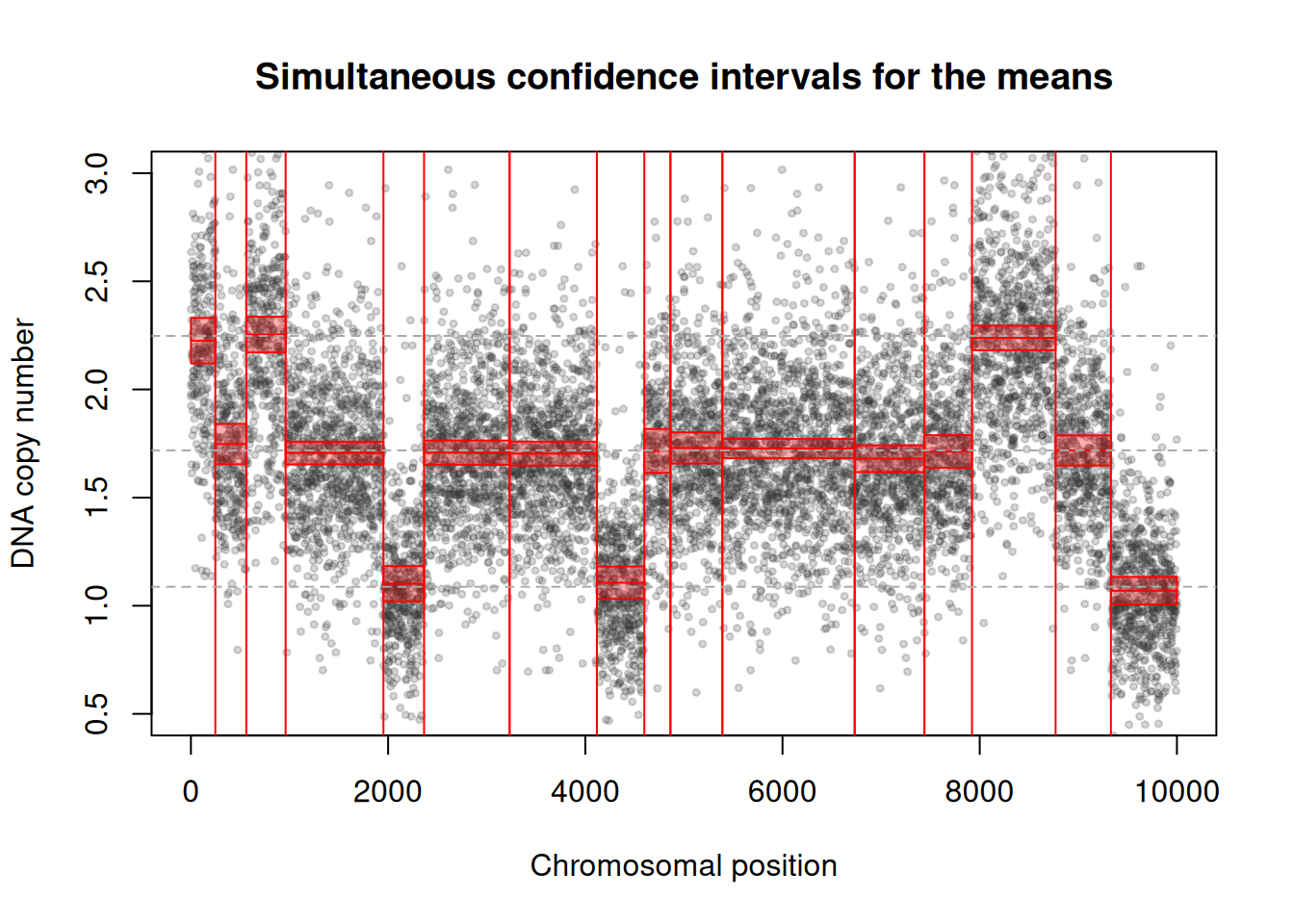}
  \end{center}
  \caption{Illustration of POSIR on DNA copy number data for a data-driven segmentation. The display is the same as in \Cref{fig:posir-CI_true-breakpoints}, except that the breakpoints have now been estimated from the data using a segmentation method, for a target number of 20 breakpoints (some of them are overlayed on the plot because they are too close on the genome). The POSIR confidence intervals are displayed as red boxes.} 
  \label{fig:posir-CI_segmentation}
\end{figure}
The simultaneous confidence intervals shown in  \Cref{fig:posir-CI_segmentation} were computed on the estimated segmentation. Here again, all intervals cover the true population mean (dashed gray lines). Moreover, the above-mentioned false positive breakpoints are easily detected by the fact that the confidence intervals associated to the flanking segments overlap. 

We note that, by construction, the confidence intervals displayed in both plots are all simultaneously valid (over regions with relative length at least $\delta$). This implies that the detection of false positive breakpoints above can come with theoretical guarantees. More precisely, consider a test procedure, on the relevance of a breakpoint, that concludes that the breakpoint is relevant if the two corresponding regions before and after the breakpoint have non-ovelapping confidence intervals. Then, if in reality a breakpoint is not relevant, that is the two regions have the same unknown average signal, the probability that the test indicates relevance is smaller than the probability that one of the two confidence intervals does not contain its target and is thus asymptotically smaller than $\alpha$. Moreover, the relevance of multiple breakpoints can be simultaneously tested, with $\alpha$ still the common type I error. This point will be further discussed in \Cref{sec:testing}.

\subsection{Comparison to a data splitting procedure}

An alternative to our POSIR method could be the following. 
The observations $y_k$ for even values of $k$ can be used to select a segmentation with your favorite method, then the independent observations $y_k$ for odd values of $k$ can be used to compute the confidence intervals.

Suppose that the segmentation procedure detected $L-1$ breakpoints, which means there are $L$ regions (intervals) in which we aim to establish simultaneous confidence intervals.
Let $I \subset \llbracket 1, n\rrbracket$ denote one of these regions. Since only the observations in even positions are used in the segmentation, we assume that its length $|I|$ is even too. At the end of this splitting procedure, the confidence interval of error level $\alpha$ for the mean value $\mu_I$ of the signal in interval $I$ is given by
\[ \CI_{I, \alpha}^{\text{split}} = \left[\widehat{\mu}_I^{\text{split}} \pm q_{1-\frac{\alpha}{2 L}} \frac{\sqrt{2} \widehat{\sigma}}{\sqrt{|I|}} \right],\, \text{with } \widehat{\mu}_I^{\text{split}} = \frac{2}{|I|} \sum_{k\in\,I,\, k\text{ odd}} y_k, \]
where $\widehat{\sigma}$ is some estimate of $\sigma$ and, following Bonferroni's rule, $q_{1-\frac{\alpha}{2 L}}$ is the quantile of level $1-\frac{\alpha}{2 L}$ of  
a Normal or a Student distribution. 
Note that the length $|I|$ is replaced by $|I|/2$ because only one observation out of two was used in the estimation.

We now provide a comparison with our proposed confidence intervals \eqref{eq:simultaneous_asymptotic_ICs}. The centerings are not the same but are similar. To simplify, we assume that the same estimate $\widehat{\sigma}$ is used on both sides, and that all selected regions are large enough that we can use the Normal distribution to compute the quantiles $q_{1-\frac{\alpha}{2 L}}$. In this setting, we can compare the diameters of the confidence intervals on both sides simply by computing their ratio 
\begin{equation} \label{eq:lengthratio}
    \frac{\sqrt{2} q_{1-\frac{\alpha}{2 L}}}{K_{1-\alpha,\delta}}.
\end{equation}
If this ratio is smaller than $1$, then the confidence intervals obtained by the splitting procedure are shorter than the ones obtained by POSIR. Since a given segmentation can lead to at most $L\leq 1/\delta$ regions of relative length $|I|/n \geq \delta$, we computed in \Cref{fig:ratios} the ratio \eqref{eq:lengthratio} with $\delta=c/L$ for various values of $L$ and of $c\in{}(0,1]$. 
\begin{figure}[htb]
  \begin{center}
    \includegraphics[width=\textwidth]{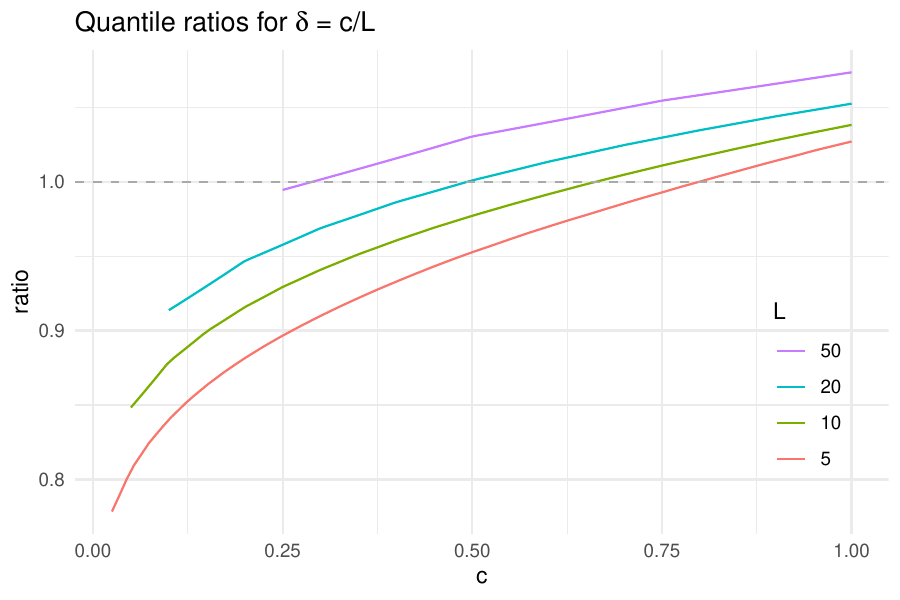}
  \end{center}
  \caption{Ratios of the diameter of confidence intervals in the splitting procedure over the diameter in POSIR. The POSIR confidence intervals are smaller whenever this ratio exceeds 1.} \label{fig:ratios}
\end{figure}

Overall, the splitting procedure for a given segmentation and POSIR provide intervals of comparable width, since all ratios are relatively close to 1. The ratios increase with the number of segments: this reflects the fact that the width of the data splitting intervals is severely affected by the number of segments. However, the main strength of the POSIR method is that the obtained confidence intervals are \emph{simultaneously valid across all possible segmentations}, for any region whose relative length is at least $\delta$ (for a pre-specified value of $\delta$). Thus, POSIR permits for instance to compare several competing segmentations, perform model selection and tests, etc. None of these is possible with the spitting procedure, which relies on establishing first a single segmentation, and whose interval sizes explicitly depend on the number of segments $L$. Given this versatility of our procedure, it is remarkable that in \Cref{fig:ratios} where only one segmentation is considered, the POSIR method remains competitive and even outperforms the data splitting approach when the segmentation selects many regions of roughly homogeneous size, close to the previously chosen $\delta$. 
    
\subsection{Further exploiting POSIR} \label{sec:testing}

We can further benefit from the validity of our simultaneous confidence intervals across all segmentations. For instance, one could cherry pick a region according to a given criterion on the associated confidence interval, among all regions that are large enough with respect to $\delta$. 
Many examples could be devised, so we limit ourselves to the following one, where the idea of cherry picking is further adapted to construct a statistical testing procedure.

Consider one-dimensional data $y=(y_1, \ldots, y_n)$, supposed to be the observations of random variables $Y=(Y_1, \ldots, Y_n)$. We want to test the null hypothesis $H_0:$ ``the variables $Y_i$, $1\leq i\leq n$, are \iid{} with $\E[Y_i]=\mu_0$'' versus any alternative $H_1$, for some known value $\mu_0$. 
Fix some $\delta\in{} (0,1)$, and consider all intervals $I$ whose relative length is at least $\delta$. Consider the following statistic
\[ T_\delta = \sup_{I: |I|\geq n\delta} \frac{\sqrt{|I|} \left|\widehat{\mu}_I -\mu_0\right|}{\widehat{\sigma}}, \]
whose distribution under $H_0$ is approximately the one of $G_\delta(B)$ if $n$ is large enough, by Theorem~\ref{thm:D_ab_asymptotics}.
This permits to compute the asymptotic p-value as 
\[ \PP(G_\delta(B) > T_\delta), \]
an estimate of which is available using the function \textrm{pposir()} of the \textrm{R} package \href{\posirpackage}{\textrm{posir}}.

\section{Proofs} \label{sec:proofs}

We first recall Donsker's theorem, which can be found for instance in \cite[theorem 16.1]{Billingsley68}.

\begin{thm}[Donsker's theorem] \label{thm:Donsker's}
  Suppose the random variables $\epsilon_k$, $k\geq 1$, are independent and identically distributed with mean $0$ and finite, positive variance $\sigma^2$. Let the random process $\widetilde{B}^n$ be defined as in \eqref{eq:def_B_tilde}. Its distribution weakly converges to Wiener measure. 
  
  Put another way, if $B$ is a standard Brownian motion on $[0, 1]$, then
  \[ \widetilde{B}^n \Lcvge B, \]
  where $\weakcvge$ denotes the weak convergence in $\cD([0, 1])$.
\end{thm}

From there, we prove the different results stated in \Cref{sec:cvrge_Wiener}.

\begin{proof}[Proof of \Cref{lem:D_ab_en_fct_de_B_tilde}] 
  First, remark that, if $k\in \llbracket 1, n\rrbracket$, then
  \[ \widetilde{B}_{k/n}^n = \frac{1}{\sigma\sqrt{n}} \sum_{i=1}^k \epsilon_i. \]
  So, if $a$ and $b$ are integers in $\llbracket 0, n\rrbracket$ such that $a<b$, then 
  \[ D_{a, b} = \frac{\widetilde{B}_t^n-\widetilde{B}_s^n}{\sqrt{t-s}} \text{ for $t=b/n$ and $s=a/n$.} \]
  Now, $a$ and $b$ have to be integers, so 
  \[ b-a \geq n\delta \Longleftrightarrow b-a \geq \lceil n\delta \rceil =: n \widetilde{\delta}_n. \]
  Since $\widetilde{B}^n$ is piecewise constant (càdlàg), see \eqref{eq:def_B_tilde}, this leads to 
  \[ \sup_{(a,b) \in \cM_{n,\delta}} \left| D_{a,b} \right| = G_{\widetilde{\delta}_n}(\widetilde{B}^n) = \sup_{0\leq s < t \leq 1 : t-s\geq\widetilde{\delta}_n} \left| \frac{\widetilde{B}_t^n-\widetilde{B}_s^n}{\sqrt{t-s}} \right|. \qedhere \]
\end{proof}

\begin{lem} \label{lem:approx_B}
  For any $\delta\in (0, 1]$, the function $G_\delta$ defined in \eqref{eq:def_G_delta} is continuous on $\cD([0, 1])$.
\end{lem}

\begin{proof}
  Let $v$ and $w$ be two elements of $\cD([0, 1])$. Then, for any $(s, t)\in E_\delta$, we have
  \begin{gather*}
    \left|\frac{v_t - v_s}{\sqrt{t-s}} \right| \leq \left|\frac{w_t - w_s}{\sqrt{t-s}} \right| + \frac{2\|v-w\|_\infty}{\sqrt{\delta}}, \\
    G_\delta(v) \leq G_\delta(w) + \frac{2\|v-w\|_\infty}{\sqrt{\delta}}.
  \end{gather*}
  By exchanging $v$ and $w$, we get 
  \[ \big| G_\delta(v) - G_\delta(w) \big| \leq \frac{2}{\sqrt{\delta}} \|v-w\|_\infty. \qedhere \]
\end{proof}

\begin{lem} \label{lem:approx_delta}
 \begin{enumerate}
   \item For any trajectory $w\in \cD([0, 1])$, the function 
   \[ (0, 1] \to \R^+, \delta \mapsto G_\delta(w) \]
   is non-increasing.
   \item Further, if $w$ is continuous, the function $\delta \mapsto G_\delta(w)$ is 
   continuous. 
 \end{enumerate}
\end{lem}

\begin{proof}
  \begin{enumerate}
    \item $G_\delta$ is defined as a supremum over $E_\delta$. Suppose $0 < \delta \leq \eta \leq 1$. Then we have $E_\eta \subset E_\delta$, and $G_\eta(w) \leq G_\delta(w)$.
    \item 
    We prove that, for any $\eta\in (0, 1]$, the function $\delta \mapsto G_\delta(w)$ is uniformly continuous on $(\eta, 1]$. \textit{A fortiori}, it is continuous on $(0, 1]$.

    So, let  $\eta\in (0, 1]$ and $\epsilon>0$. Since $w$ is continuous on the compact set $[0, 1]$, it is also uniformly continuous: 
    \[ \exists \gamma >0, \forall s, t \in [0, 1], |t-s|\leq \gamma \Rightarrow |w_t-w_s| \leq \frac{\sqrt{\eta}}{4} \epsilon. \]
    Moreover, since 
    \[ \lim_{t\to 0^+} \sqrt{\frac{\eta}{\eta+t}} = 1, \]
    let us choose $\gamma$ small enough so that 
    \[ \left| \sqrt{\frac{\eta}{\eta+\gamma}} -1 \right| \leq \frac{\sqrt{\eta}\, \epsilon}{4 \|w\|_\infty}. \]

    Now, let $\delta \leq \mu \in (\eta, 1)$ such that $\mu-\delta\leq \gamma$. 
    Since $w$ is continuous, the function
    \[ E_\delta \to \R,\, (s, t) \mapsto \left|\frac{w_t - w_s}{\sqrt{t-s}} \right| \]
    is also continuous, and $E_\delta$ is a compact set. Therefore, there exists a couple $(u, v)\in E_\delta$ such that
    \[ G_\delta(w) = \left|\frac{w_v - w_u}{\sqrt{v-u}} \right|. \]
    Note that $u<v$ since $(u, v)\in E_\delta$.

    Suppose $v-u>\mu$. Then 
    \[ G_\delta(w) = \left|\frac{w_v - w_u}{\sqrt{v-u}} \right| \leq G_{\mu}(w) \leq G_\delta(w). \]

    Otherwise, we have $\delta\leq v-u\leq \mu$. Let $\widetilde{u}=\max(0, v-\mu)$, and $\widetilde{v}=\widetilde{u}+\mu=\max(v, \mu)$. This ensures that $\widetilde{u}\in [0, u]$, $\widetilde{v}\in [v, 1]$, and $\max(u-\widetilde{u}, \widetilde{v}-v)\leq \mu-\delta\leq \gamma$. Now, we have
    \begin{align*}
      G_\delta(w) \geq G_\mu(w) &\geq \left|\frac{w_{\widetilde{v}} - w_{\widetilde{u}}}{\sqrt{{\widetilde{v}}-{\widetilde{u}}}} \right| %\\ &
      \geq \frac{|w_v - w_u|}{\sqrt{\mu}} - \frac{2}{\sqrt{\eta}} \sup_{|t-s|\leq \gamma} |w_t-w_s| \\ &
      \geq \sqrt{\frac{\delta}{\delta+\gamma}}\, G_\delta(w)-\frac{\epsilon}{2}, 
    \end{align*}
    since $\mu\geq\delta\geq\eta$. 
    Further, we have 
    \[ \sqrt{\frac{\delta}{\delta+\gamma}} \geq \sqrt{\frac{\eta}{\eta+\gamma}} \geq 1 - \left| \sqrt{\frac{\eta}{\eta+\gamma}} -1 \right|, \]
    so that 
    \[ \sqrt{\frac{\delta}{\delta+\gamma}}\, G_\delta(w) \geq G_\delta(w) - \frac{\sqrt{\eta}\, G_\delta(w)}{4 \|w\|_\infty}\, \epsilon \geq G_\delta(w) - \frac{\epsilon}{2}, \]
    since $G_\delta(w) \leq \frac{2 \|w\|_\infty}{\sqrt{\delta}}$ and $\delta\geq\eta$.
    
    At the end, we obtain the expected $|G_\mu(w) - G_\delta(w)|\leq \epsilon$. \qedhere
  \end{enumerate}
\end{proof}

\begin{proof}[Proof of \Cref{thm:D_ab_asymptotics}]
  Let $B$ be a Brownian motion on $[0, 1]$. From \Cref{thm:Donsker's} and \Cref{lem:approx_B}, we already have $\displaystyle G_{\delta}(\widetilde{B}^n) \Lcvge G_\delta(B)$ for any $\delta\in (0, 1]$. 
  
  In order to obtain the same asymptotics for $G_{\widetilde{\delta}_n}(\widetilde{B}^n)$, it is sufficient to prove that its cumulative distribution function point-wise converges to the one of $G_\delta(B)$.
  
  Let us first consider the case $\delta=1$. Then $\widetilde{\delta}_n=\delta$ and \Cref{thm:D_ab_asymptotics} holds.
  
  Let us now fix some $\delta\in (0, 1)$, as well as $t\in\R$ and $\epsilon>0$. 
  By \Cref{lem:approx_delta}, \[ G_{\eta}(B) \mathop{\longrightarrow}_{\eta\to\delta^+}^{a.s.} G_{\delta}(B), \]
  since $B$ is \as{} continuous. It entails the convergence in distribution. Therefore, there exists $\eta \in (\delta, 1]$ such as
  \[ \big| \Pr\big(G_{\eta}(B)\leq t\big) - \Pr\big(G_{\delta}(B)\leq t\big) \big| \leq \epsilon. \]
  
  Moreover, since $\widetilde{\delta}_n \to \delta$ as $n\to\infty$, for $n$ large enough we have
  \begin{gather*}
    \delta \leq \widetilde{\delta}_n \leq \eta \\
    G_{\eta}(\widetilde{B}^n) \leq G_{\widetilde{\delta}_n}(\widetilde{B}^n) \leq G_{\delta}(\widetilde{B}^n) \\
    \Pr\big(G_{\delta}(\widetilde{B}^n)\leq t\big) \leq \Pr\big(G_{\widetilde{\delta}_n}(\widetilde{B}^n)\leq t\big) \leq \Pr\big(G_{\eta}(\widetilde{B}^n)\leq t\big),
  \end{gather*}
  as well as 
  \[ G_{\eta}(B) \leq G_{\delta}(B) \text{ and } \Pr\big(G_{\delta}(B)\leq t\big) \leq \Pr\big(G_{\eta}(B)\leq t\big). \]
  
  Going back to the convergence in distribution of $G_{\delta}(\widetilde{B}^n)$ towards $G_{\delta}(B)$  for any $\delta\in (0, 1]$, for $n$ large enough we have both
  \begin{gather*}
    \big| \Pr\big(G_{\delta}(\widetilde{B}^n)\leq t\big) - \Pr\big(G_{\delta}(B)\leq t\big) \big| \leq \epsilon, \\
    \big| \Pr\big(G_{\eta}(\widetilde{B}^n)\leq t\big) - \Pr\big(G_{\eta}(B)\leq t\big) \big| \leq \epsilon.
  \end{gather*}
  
  Simple triangle inequalities eventually give
  \[ \big| \Pr\big(G_{\widetilde{\delta}_n}(\widetilde{B}^n)\leq t\big) - \Pr\big(G_{\delta}(B)\leq t\big) \big| \leq 2 \epsilon. \]
  
  We have the desired point-wise convergence of the cumulative distribution functions, which proves $\displaystyle G_{\widetilde{\delta}_n}(\widetilde{B}^n) \Lcvge G_\delta(B)$.
\end{proof}

\begin{proof}[Proof of \Cref{thm:simultaneous_asymptotic_ICs}]
  It follows from \Cref{thm:D_ab_asymptotics} and Slutsky’s theorem.
\end{proof}

\begin{proof}[Proof of \Cref{corol:simultaneous_selected_asymptotic_ICs}]
  It comes from \eqref{eq:PoSI_for_regions} and \Cref{thm:simultaneous_asymptotic_ICs}.
\end{proof}

\section{Extension to higher dimensions} \label{sec:higherdim}

In higher dimensions we follow a pattern very similar to the one-dimensional case. We begin with an extension of \Cref{thm:Donsker's} coming from \cite[Corollary 1]{Wichura69} (see also \cite{Arras2020}). It is stated in the space $\cD\big([0, 1]^d\big)$ of all cadlag functions on $[0, 1]^d$, defined in \cite{Wichura68}. $\cD\big([0, 1]^d\big)$ is equipped with the supremum norm and, below in \Cref{thm:Donsker_Brownian_sheet}, the symbol $\weakcvge$ denotes the weak convergence for the associated topology.

\begin{thm}[Donsker's theorem to Brownian sheets] \label{thm:Donsker_Brownian_sheet}
  Let $d\geq 1$ be the dimension, and let $B^{(d)}$ be a Brownian sheet on $[0, 1]^d$, \ie{} $B^{(d)}=\big(B^{(d)}_\mathbf{t}\big)_{\mathbf{t}\in [0, 1]^d}$ is a real-valued, \as{} continuous, Gaussian process with mean zero and covariance
  \[ \cov(B^{(d)}_\mathbf{t}, B^{(d)}_\mathbf{s}) = \prod_{j=1}^d \min(t_j, s_j), \]
  where $\mathbf{t}=(t_1, \ldots, t_d) \in [0, 1]^d$ and $\mathbf{s}=(s_1, \ldots, s_d) \in [0, 1]^d$. 
  Let also $\epsilon_\mathbf{k}$, $\mathbf{k}=(k_1, \ldots, k_d) \in (\N^*)^d$, be a collection of \iid{} centered random variables with finite positive variance $\sigma^2$.
  For any $\mathbf{n}=(n_1, \ldots, n_d) \in (\N^*)^d$ and $\mathbf{t}\in [0, 1]^d$, define 
  \[ S_\mathbf{t}^\mathbf{n} = \{ \mathbf{k}\in (\N^*)^d : \forall j \in \llbracket 1, d\rrbracket, k_j\leq \lfloor t_j n_j\rfloor \}, \]  
  and 
  \begin{equation} \label{eq:def_B_n_higher_dim}
    \widetilde{B}_\mathbf{t}^\mathbf{n} = \left( \sigma \prod_{j=1}^d \sqrt{n_j} \right)^{-1} \sum_{\mathbf{k}\in S_\mathbf{t}^\mathbf{n}} \epsilon_\mathbf{k}.
  \end{equation}
  Then, \emph{if $\forall j \in \llbracket 1, d\rrbracket$, $n_j\to\infty$,} we get
  \[ \widetilde{B}^\mathbf{n} \weakcvge B^{(d)}. \]
\end{thm}

\subsection{Rectangular regions}

Let us consider the partial order relation on $\N^d$ defined, for any $\mathbf{a}, \mathbf{b} \in \N^d$, by
\[ \mathbf{a}\leq \mathbf{b} \Longleftrightarrow \forall j\in \llbracket 1, d\rrbracket,\, a_j\leq b_j. \]
We also note $\mathbf{a}+\mathbf{1} = (a_j+1)_{1\leq j\leq d}$ and, for any $\mathbf{h}\in\{0, 1\}^d$, $|\mathbf{h}| = \sum_{j=1}^d h_j$ and $\mathbf{a}+\mathbf{h}.\mathbf{b} = (a_j+h_j b_j)_{1\leq j\leq d}$. This last notation naturally extends to $\R^d$.

Suppose we have $\mathbf{a}+\mathbf{1} \leq \mathbf{b} \leq \mathbf{n}$, and define the rectangular region $S_{(\mathbf{a}, \mathbf{b})}$ by
\[
  S_{(\mathbf{a}, \mathbf{b})} 
  = \prod_{j=1}^d \llbracket a_j+1, b_j\rrbracket
  = \big\{ \mathbf{k}\in (\N^*)^d : \mathbf{a}+\mathbf{1} \leq \mathbf{k} \leq \mathbf{b} \big\}.
\]
Consider the random variable
\begin{equation}\label{eq:def_Dab_higher_dim}
  D_{(\mathbf{a}, \mathbf{b})} = \left( \sigma \prod_{j=1}^d \sqrt{b_j-a_j} \right)^{-1} \sum_{\mathbf{k}\in S_{(\mathbf{a}, \mathbf{b})}} \epsilon_\mathbf{k}.
\end{equation}
Similarly to \Cref{lem:D_ab_en_fct_de_B_tilde}, we aim at establishing a link between \eqref{eq:def_Dab_higher_dim} and \eqref{eq:def_B_n_higher_dim}. 
The following lemma may be known by experts, but we provide a full statement and proof of self-containedness.
\begin{lem} \label{lem:cumulativesumsmultidim}
  Let $d\in\N^*$ be a positive integer. Let $\mathbf{a}, \mathbf{b} \in \N^d$ be such that $\mathbf{a}+\mathbf{1} \leq \mathbf{b}$, and pose $S_{(\mathbf{a}, \mathbf{b})} = \prod_{j=1}^d \llbracket a_j+1, b_j\rrbracket$. Then, for any function $f : (\N^*)^d\to \R$, we have
  \[ \sum_{\mathbf{k}\in S_{(\mathbf{a}, \mathbf{b})}} f(\mathbf{k}) = \sum_{\mathbf{h}\in\{0, 1\}^d} (-1)^{d-|\mathbf{h}|} \sum_{\mathbf{k}\in(\N^*)^d : \mathbf{k}\leq \mathbf{a}+\mathbf{h}.(\mathbf{b}-\mathbf{a})} f(\mathbf{k}). \]
\end{lem}
So the sum over a rectangular region is a linear combination of the cumulative sums taken at the vertices of the rectangle.

\begin{proof}
  We use an induction on $d$. With $d=1$, we indeed have
  \[ \sum_{k=a+1}^b f(k) = \sum_{k=1}^b f(k) - \sum_{k=1}^a f(k).\]

  Suppose now that the formula holds for a given $d\geq 1$. Let $\mathbf{a}, \mathbf{b} \in \N^{d+1}$ such that $\mathbf{a}+\mathbf{1} \leq \mathbf{b}$. Let $S_{(\mathbf{a}, \mathbf{b})} = \prod_{j=1}^{d+1} \llbracket a_j+1, b_j\rrbracket$ be defined as above, but consider also
  \[ \widetilde{\mathbf{a}} = (a_j)_{1\leq j\leq d}, \quad \widetilde{\mathbf{b}} = (b_j)_{1\leq j\leq d}, \quad \widetilde{S}_{(\mathbf{a}, \mathbf{b})} = \prod_{j=1}^d [a_j+1, b_j], \]
  which are the projections of $\mathbf{a}$, $\mathbf{b}$, and $S_{(\mathbf{a}, \mathbf{b})}$ on $\N^d$.

  Consider a function $f : (\N^*)^{d+1}\to \R$. 
  Define the set of functions $\{ f_k : (\N^*)^d\to \R , k \in \N^*\}$, by $f(\mathbf{k}) = f_{k_{d+1}}(\widetilde{\mathbf{k}})$ for all $\mathbf{k} \in (\N^*)^{d+1}$. Now, we go back from the end.
  \begin{multline*}
    \sum_{\mathbf{h}\in\{0, 1\}^{d+1}} (-1)^{d+1-|\mathbf{h}|} \sum_{\mathbf{k}\in(\N^*)^{d+1} : \mathbf{k}\leq \mathbf{a}+\mathbf{h}.(\mathbf{b}-\mathbf{a})} f(\mathbf{k}) \\
    \begin{aligned}
      &= \sum_{\mathbf{h}\in\{0, 1\}^{d+1}} (-1)^{d+1-|\mathbf{h}|} \sum_{\widetilde{\mathbf{k}}\in(\N^*)^d : \widetilde{\mathbf{k}}\leq \widetilde{\mathbf{a}}+\widetilde{\mathbf{h}}.(\widetilde{\mathbf{b}}-\widetilde{\mathbf{a}})} \sum_{l=1}^{a_{d+1}+(b_{d+1}-a_{d+1})h_{d+1}} f_l(\widetilde{\mathbf{k}}) \\
      &= \left[\sum_{l=1}^{b_{d+1}} - \sum_{l=1}^{a_{d+1}}\right] 
      \sum_{\widetilde{\mathbf{h}}\in\{0, 1\}^d} (-1)^{d-|\widetilde{\mathbf{h}}|} \sum_{\widetilde{\mathbf{k}}\in(\N^*)^d : \widetilde{\mathbf{k}}\leq \widetilde{\mathbf{a}}+\widetilde{\mathbf{h}}.(\widetilde{\mathbf{b}}-\widetilde{\mathbf{a}})} f_l(\widetilde{\mathbf{k}}) \\
      &= \sum_{l=a_{d+1}+1}^{b_{d+1}} \sum_{\widetilde{\mathbf{k}}\in \widetilde{S}_{(\mathbf{a}, \mathbf{b})}} f_l(\widetilde{\mathbf{k}}) 
      = \sum_{\mathbf{k}\in S_{(\mathbf{a}, \mathbf{b})}} f(\mathbf{k}).
    \end{aligned}
  \end{multline*}
  Therefore the formula holds for any $d\in\N^*$.
\end{proof}

\begin{corol} \label{corol:Dab_cumsums_multidim}
  Further define $\widetilde{\mathbf{t}}(\mathbf{h})\in [0, 1]^d$ as
  \[ \widetilde{\mathbf{t}}(\mathbf{h}) = \left( \frac{a_j+(b_j-a_j)h_j}{n_j} \right)_{1\leq j\leq d}. \]
  Then 
  \[ D_{(\mathbf{a}, \mathbf{b})} = \left(\prod_{j=1}^d \sqrt{\frac{n_j}{b_j-a_j}}\right) \sum_{\mathbf{h}\in \{0, 1\}^d} (-1)^{d-|\mathbf{h}|} \widetilde{B}^\mathbf{n}_{\widetilde{\mathbf{t}}(\mathbf{h})}. \]
\end{corol}

\begin{proof}
  Begin with \eqref{eq:def_Dab_higher_dim} and apply \Cref{lem:cumulativesumsmultidim} to $\sum_{\mathbf{k}\in S_{(\mathbf{a}, \mathbf{b})}} \epsilon_\mathbf{k}$. The rest follows from the equivalence \[\mathbf{k}\in S_{\widetilde{\mathbf{t}}(\mathbf{h})}^\mathbf{n} \Longleftrightarrow \mathbf{k}\leq \mathbf{a}+\mathbf{h}.(\mathbf{b}-\mathbf{a})\] and the definition \eqref{eq:def_B_n_higher_dim} of $\widetilde{B}^\mathbf{n}$.
\end{proof}

\subsection{Families of rectangular regions}

Several families of rectangular regions can be considered. In practice one has to fix in advance the family of regions among which the selection procedure can choose. 

As an example of such a possible family, let $\delta\in (0, 1]$ be fixed, and define
\[ \cM_{\mathbf{n}, \delta}^{d, 1} = \{ (\mathbf{a}, \mathbf{b}) \in (\N^{d})^2 : \forall j \in \llbracket 1, d\rrbracket,\, a_j+\delta n_j \leq b_j \leq n_j \} \]
and
\[ E_\delta^{d, 1} = \{ (\mathbf{s}, \mathbf{t}) \in ([0, 1]^d)^2 : \forall j \in \llbracket 1, d\rrbracket,\, s_j+\delta \leq t_j \}. \]

Then, a two-dimensional extension of \Cref{thm:D_ab_asymptotics}, adapted to the above family of regions, is the following.

\begin{thm} \label{thm:multidim_BaBoNe_processes}
  Let $B^{(d)}$ be a Brownian sheet on $[0, 1]^d$, and fix $\delta\in (0, 1]$. Then, as long as $\forall j \in \llbracket 1, d\rrbracket$, $n_j\to\infty$, we have
  \[ \sup_{(\mathbf{a}, \mathbf{b}) \in \cM_{\mathbf{n}, \delta}^{d, 1}} \left| D_{(\mathbf{a}, \mathbf{b})} \right| \weakcvge \sup_{(\mathbf{s}, \mathbf{t})\in E_\delta^{d, 1}} \left( \prod_{j=1}^d \sqrt{t_j-s_j} \right)^{-1} \left|\sum_{\mathbf{h}\in \{0, 1\}^d} (-1)^{d-|\mathbf{h}|}\, B^{(d)}_{\mathbf{s}+\mathbf{h}.(\mathbf{t}-\mathbf{s})} \right|. \]
\end{thm}
From there, simultaneous confidence intervals can be built, as was done previously in \Cref{thm:simultaneous_asymptotic_ICs} and \Cref{corol:simultaneous_selected_asymptotic_ICs}.

The main deviation of the proof of \Cref{thm:multidim_BaBoNe_processes} from the one of \Cref{thm:D_ab_asymptotics} is to be found in the approximation of $\delta$ by $\widetilde{\delta}_n$ in \Cref{lem:D_ab_en_fct_de_B_tilde}. In higher dimension the approximation must be done separately in each coordinate, unless $\mathbf{n} = n \mathbf{1}$ (that is, unless $\mathbf{n}$ grows at the same rate in all directions).

This leads us to define a secondary family of rectangular regions. 
Let $\boldsymbol{\gamma}\in (0, 1]^d$ be fixed, and define
\[ \cM_{\mathbf{n},\boldsymbol{\gamma}}^{d, 2} = \{ (\mathbf{a}, \mathbf{b}) \in (\N^{d})^2 : \forall j \in \llbracket 1, d\rrbracket,\, a_j+ \gamma_j n_j \leq b_j \leq n_j \} \]
and
\[ E_{\boldsymbol{\gamma}}^{d, 2} = \{ (\mathbf{s}, \mathbf{t}) \in ([0, 1]^d)^2 : \forall j \in \llbracket 1, d\rrbracket,\, s_j+\gamma_j \leq t_j \}. \]
Note that $\cM_{\mathbf{n}, \delta}^{d, 1} = \cM_{\mathbf{n}, \delta \mathbf{1}}^{d, 2}$ and $ E_\delta^{d, 1} = E_{\delta \mathbf{1}}^{d, 2}$. Variants of \Cref{thm:multidim_BaBoNe_processes} could easily be stated on the basis of this secondary family of rectangular regions.

Let also $G_{\boldsymbol{\gamma}}^{d, 2} : \cD\big([0, 1]^d\big) \to \R^+$ be defined by
\begin{equation} \label{eq:def_G_delta_multidim}
  G_{\boldsymbol{\gamma}}^{d, 2}\big(w^{(d)}\big) := \sup_{(\mathbf{s}, \mathbf{t})\in E_{\boldsymbol{\gamma}}^{d, 2}} \left( \prod_{j=1}^d \sqrt{t_j-s_j} \right)^{-1} \left|\sum_{\mathbf{h}\in \{0, 1\}^d} (-1)^{d-|\mathbf{h}|}\, w^{(d)}_{\mathbf{s}+\mathbf{h}.(\mathbf{t}-\mathbf{s})} \right|,
\end{equation}
so the right-hand limit in \Cref{thm:multidim_BaBoNe_processes} is $G_{\delta \mathbf{1}}^{d, 2}\big(B^{(d)}\big)$.

\begin{lem} \label{lem:Dab_deltatilde_multidim}
  \[ \sup_{(\mathbf{a}, \mathbf{b}) \in \cM_{\mathbf{n}, \delta}^{d, 1}} \left| D_{(\mathbf{a}, \mathbf{b})} \right| = G_{\widetilde{{\boldsymbol{\gamma}}}_\mathbf{n}}^{d, 2}\big(\widetilde{B}^\mathbf{n}\big), \]
  where $\widetilde{{\boldsymbol{\gamma}}}_\mathbf{n} = \left(\dfrac{\lceil \delta n_j \rceil}{n_j}\right)_{1\leq j\leq d}$. We therefore have $\delta \mathbf{1} \leq \widetilde{{\boldsymbol{\gamma}}}_\mathbf{n}$ and $\mathbf{n}.\big(\widetilde{{\boldsymbol{\gamma}}}_\mathbf{n}-\delta \mathbf{1}\big)\leq\mathbf{1}$.
\end{lem}

\Cref{lem:Dab_deltatilde_multidim} is a straightforward extension of \Cref{lem:D_ab_en_fct_de_B_tilde} using \Cref{corol:Dab_cumsums_multidim}. We will not give the details neither for the rest of the proof of \Cref{thm:multidim_BaBoNe_processes}, which can be similarly adapted from the proof of \Cref{thm:D_ab_asymptotics}. We have to check that 
\begin{itemize}
    \item for all ${\boldsymbol{\gamma}}\in (0, 1]^d$, $G_{\boldsymbol{\gamma}}^{d, 2}$ is continuous on $\cD\big([0, 1]^d\big)$, 
    \item for any sheet $w^{(d)}\in\cD\big([0, 1]^d\big)$, ${\boldsymbol{\gamma}}_1\leq {\boldsymbol{\gamma}}_2 \Longrightarrow G_{{\boldsymbol{\gamma}}_1}^{d, 2}\big(w^{(d)}\big) \geq G_{{\boldsymbol{\gamma}}_2}^{d, 2}\big(w^{(d)}\big)$,
    \item for any continuous sheet $w^{(d)}$, the function ${\boldsymbol{\gamma}}\mapsto G_{\boldsymbol{\gamma}}^{d, 2}\big(w^{(d)}\big)$ is continuous on $(0, 1]^d$,
\end{itemize}
and conclude using triangular inegalities and Donsker's \Cref{thm:Donsker_Brownian_sheet}.

\subsection{Simulations in dimension 2} \label{sec:2Dsimulations}

Our simulations in dimension 2 follow a similar framework as the ones in dimension 1, see \Cref{sec:1Dsimulations}. Again, all Monte-Carlo estimates use $10^6$ independent trajectories of (simulated) 2D POSIR processes.

First, we computed the quantiles of the 2D POSIR process $G_{\delta \mathbf{1}}^{2, 2}\big(B^{(2)}\big)$ appearing in \Cref{thm:multidim_BaBoNe_processes}. Due to heavy computational costs, we limited ourselves to a discretization number $n=400$ in each direction. It still means that we simulated $1.6 \times 10^5$ independent standard variables $\epsilon_k$ for each trajectory. See \Cref{table:quantiles_2D} for the results.

\begin{table}[hptb]
  \begin{center}
    \small
    \begin{tabular}{|r|cccccccccc|}
      \hline
      $\alpha\backslash\delta$ & 1 & 0.9 & 0.8 & 0.7 & 0.6 & 0.5 & 0.4 & 0.3 & 0.2 & 0.1 \\ 
      \hline
      .5 & 0.675 & 1.480 & 1.888 & 2.233 & 2.552 & 2.867 & 3.189 & 3.541 & 3.950 & 4.475 \\ 
      .2 & 1.282 & 2.070 & 2.461 & 2.787 & 3.087 & 3.380 & 3.677 & 3.996 & 4.360 & 4.823 \\ 
      .1 & 1.644 & 2.425 & 2.805 & 3.118 & 3.405 & 3.682 & 3.963 & 4.263 & 4.599 & 5.032 \\ 
      .05 & 1.959 & 2.732 & 3.104 & 3.406 & 3.681 & 3.946 & 4.214 & 4.495 & 4.812 & 5.219 \\ 
      .01 & 2.577 & 3.337 & 3.691 & 3.973 & 4.227 & 4.468 & 4.707 & 4.958 & 5.244 & 5.607 \\ 
      .005 & 2.810 & 3.562 & 3.910 & 4.185 & 4.429 & 4.661 & 4.893 & 5.137 & 5.413 & 5.757 \\ 
      .001 & 3.296 & 4.025 & 4.363 & 4.628 & 4.864 & 5.071 & 5.294 & 5.521 & 5.766 & 6.086 \\
      \hline
    \end{tabular}
  \end{center}
  \caption{Some empirical quantiles for the 2D POSIR process.}
  \label{table:quantiles_2D}
\end{table}

Theroretically, in 2D as previously in 1D, discretization of the continuous POSIR process introduces an unknown negative bias, especially for $\delta$ small. This would lead to slightly smaller confidence intervals, and thus to slightly larger effective simultaneous error levels.

However, the estimates of effective error levels presented below strongly suggest that this supposed bias in the quantiles is indeed negligible, at least for $\delta \ge 0.3$.

\begin{figure}[htb]
  \includegraphics[width=\textwidth]{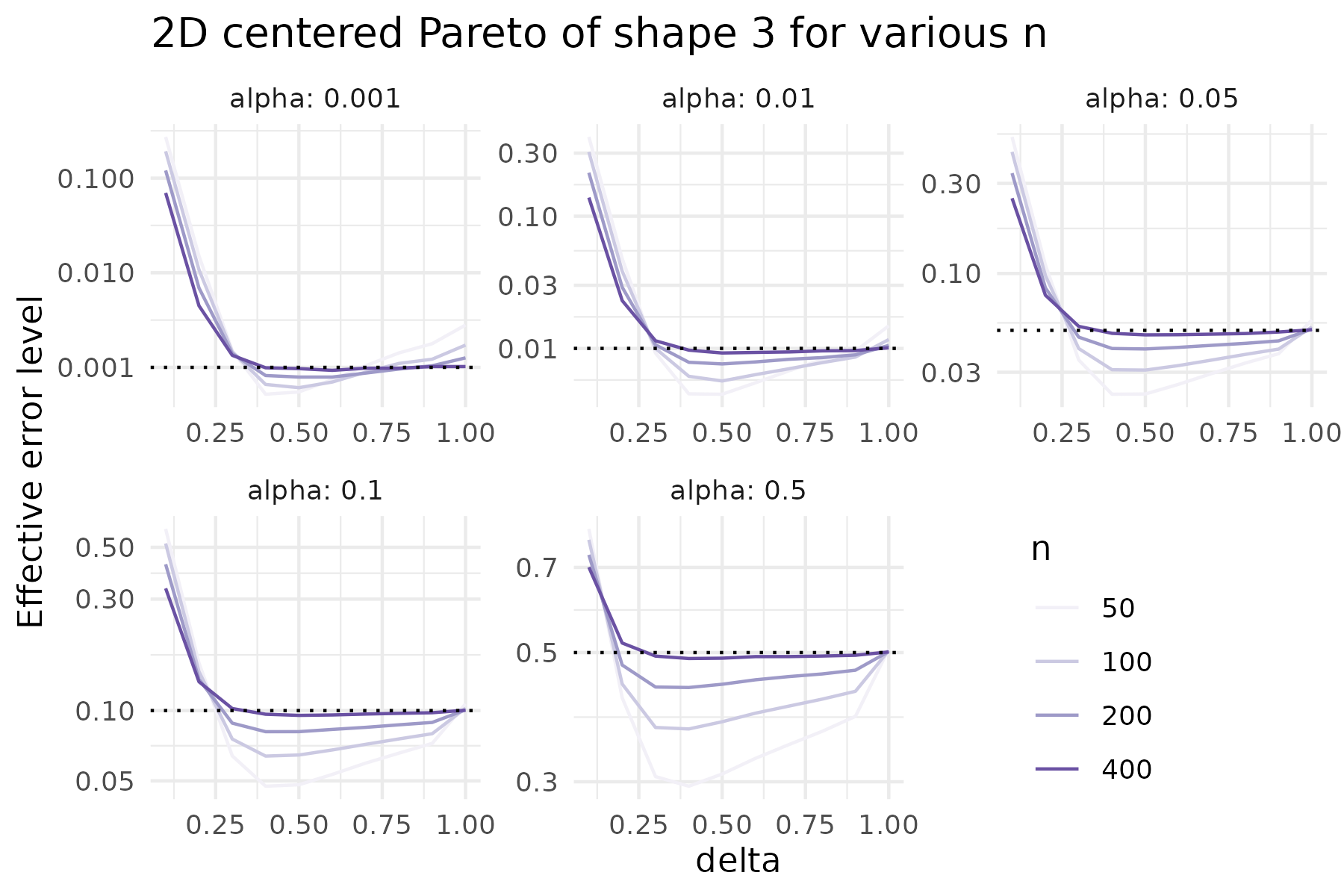}
  \vspace*{-1cm}
  \caption{Effective simultaneous error levels in dimension 2 for the Pareto distribution with shape parameter $3$ and various $n$.}
  \label{fig:eff-conf-level_2D_centered-pareto_3}
\end{figure}

Indeed, we estimated the effective error levels using again the Pareto distribution with shape parameter $3$ for the variables $\epsilon_k$, at various values of $n$. The results, which can be found in \Cref{fig:eff-conf-level_2D_centered-pareto_3}, are rather better than in dimension 1 as long as $\delta$ does not go below $0.3$. In particular, the asymptotics in $n$ are much better. 

\section{Conclusion}

We have provided confidence intervals on the average signal of a region, that account for the fact that this region is previously selected in a data-driven way. These intervals are based on the process mapping all possible regions to their estimation errors, for which we prove the convergence to an explicit limiting Gaussian process. This asymptotic approach allows for any model selection procedure, and any noise distribution with finite variance. We also note that, in a straightforward manner, our confidence intervals can yield statistical tests, for instance on the comparison between regions, that are also valid in a post-selection context.

In practice, the data-set size needs not be large for the empirical coverage proportions to already reach the nominal level. Similarly, noise distributions with heavy tails do not overly damage these coverage proportions. Compared to alternative approaches such as data splitting, our procedure provides the most freedom, in that even multiple confidence intervals for multiple (and even adaptively chosen) regions come with asymptotic post-selection guarantees. 

A possible extension of our work could rely on a better theoretical understanding of the underlying Gaussian process. For instance, exploiting the scale invariance of the Brownian motion may lead to sharper estimates of the POSIR quantiles for small regions, and reduce their computational cost. It may also prove beneficial in higher dimensions.

Another relevant, though challenging, possible extension of this work is to allow for more complex regions than rectangles. For instance, in order to cover more practical scenarios, one may be interested in allowing non-parametric regions, or parametric ones with a number of parameters diverging with the data-set size. From the theoretical viewpoint, extending our guarantees to these situations would necessitate to study random processes indexed by complex sets (for instance a set of ``smooth'' regions), or with index set changing with $n$.
Even if a limiting Gaussian process was obtained, more complex selected regions would also yield computational challenges in estimating the quantiles of the supremum of this process. Dedicated, more advanced, methods to this analysis of supremum could thus be required.

\paragraph{Acknowledgements.}
This work was supported by the Project GAP (ANR-21-CE40-0007) of the French National Research Agency (ANR).

\newcommand{\etalchar}[1]{$^{#1}$}

\appendix

\section{Additional numerical results}
In \Cref{fig:eff-conf-level_centered-pareto_3}, in the frame of \Cref{subsection:effective:confidence}, we provide for completeness the effective simultaneous error levels for the Pareto distribution with shape parameter $3$ and various $n$. As expected, the effective error levels are substantially closer to the nominal level than for the Pareto (2.1) distribution (see \Cref{fig:eff-conf-level_centered-pareto_2.1}). 
\begin{figure}[hptb]
  \includegraphics[width=\textwidth]{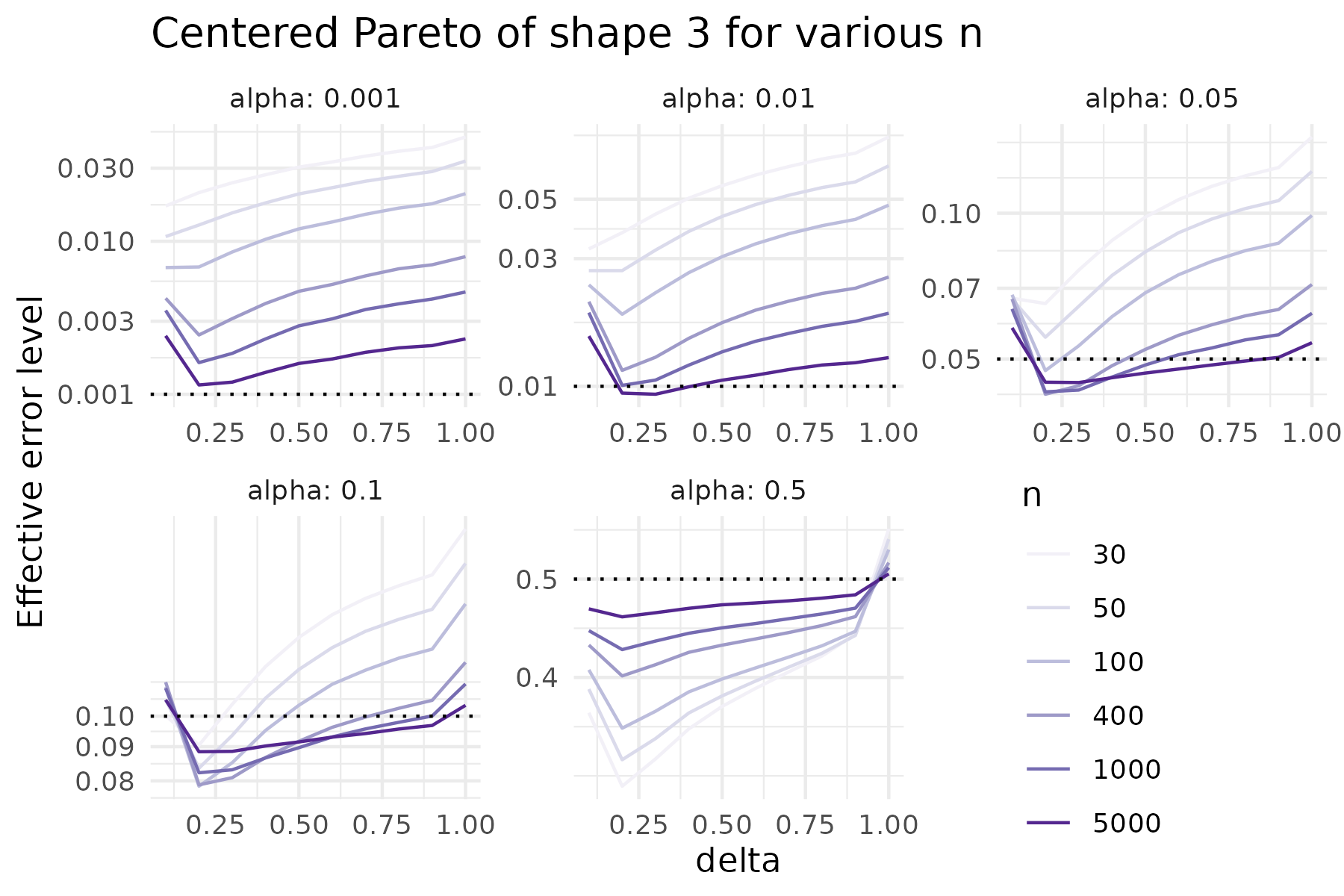}
  \vspace*{-1cm}
  \caption{Effective simultaneous error levels for the Pareto distribution with shape parameter $3$ and various $n$.}
  \label{fig:eff-conf-level_centered-pareto_3}
\end{figure}

\end{document}